\documentclass[a4paper,11pt]{article} 
\usepackage[italian,english]{babel}			
\usepackage{amsmath,amssymb,amsthm}
\usepackage[normalem]{ulem}
\usepackage{esint}
\usepackage{amsfonts} 
\usepackage{graphicx} 
\usepackage{braket}
\usepackage{times}				
\usepackage{ifthen}
\usepackage{xspace}
\usepackage{fancyhdr}
\usepackage{mathrsfs}
\usepackage{enumerate}
\usepackage{verbatim}
\usepackage{hyperref}
\usepackage{relsize}
\usepackage[latin1]{inputenc}

\usepackage{color} 
\allowdisplaybreaks 
\definecolor{Red}{cmyk}{0,1,1,0.2}

\def \N{\mathbb{N}}
\def \R{\mathbb{R}}
\def \E{\mathbb{E}}
\def \P{\mathbb{P}}
\newcommand{\X}{{\R^d}}

\newcommand{\bx}{\bar{x}}
\newcommand{\by}{\bar{y}}

\newcommand{\bt}{\bar{t}}
\usepackage{fancyhdr}

\usepackage{fancyhdr}
\newtheorem{definition}{Definition}[section] 
\theoremstyle{definition}

\theoremstyle{remark}
\newtheorem{remark}{Remark}[section]
\theoremstyle{plain}
\newtheorem{theorem}{Theorem}[section]
\newtheorem{lemma}{Lemma}[section]
\newtheorem{proposition}{Proposition}[section]

\newtheorem{example}{Example}[section]
\newtheorem{hp}{Hypotheses}[section]

\numberwithin{equation}{section}
 
\title{}

\newcommand{\xz}{\color{black}}

\date{}

\begin{document}
	\title{Random Time Dynamical Systems }
	\author{
	\textbf{Rossana Capuani}\\
 Department of Mathematics, University of Trento\\
 Trento, Italy\\
 Email: rossana.capuani@univr.it
\and
\textbf{Luca Di Persio}\\
 Department of Computer Science-Math.College, University of Verona\\
 I-37134 Verona, Italy\\
 Email: luca.dipersio@univr.it
\and
\textbf{Yuri Kondratiev}\\
 Department of Mathematics, University of Bielefeld\\
 D-33615 Bielefeld, Germany\\
 Dragomanov University, Kiev, Ukraine\\
 Email: kondrat@math.uni-bielefeld.de
\and
\textbf{Michele Ricciardi}\\
Department of Computer Science-Math.College, University of Verona\\
  I-37134 Verona, Italy\\
  Email: michele.ricciardi@univr.it
\and
\textbf{Jos{\'e} Lu{\'i}s
da Silva},\\
 CIMA, University of Madeira, Campus da Penteada\\
 9020-105 Funchal, Portugal\\
 Email: joses@staff.uma.pt
}

\date{\today}

\maketitle

\begin{abstract}
 In this paper, we introduce the concept of  random time changes in  dynamical systems. 
 The subordination
principle may be applied 
to study the long time behavior of the random time systems. We show,
under certain assumptions on the class of random time, that the subordinated
system exhibits a slower time decay which is determined by the random
time characteristics.  Along the path asymptotic, a random time change is reflected in
the new velocity of the resulting dynamics. \\

	\noindent \textit{Keywords}: Dynamical systems, random time change,
inverse subordinator, long time behavior \\

	\noindent \textbf{MSC Subject classifications}:  35R11, 37A50, 45M05,
 46N30, 60G20, 60G52. \\
\end{abstract}

	\section{Introduction}
	The idea to consider stochastic processes with general random times
is known at least from the classical book by Gikhman and Skorokhod
\cite{GS74}. In the case of Markov processes time changes by subordinators
has been already considered by Bochner in \cite{Bochner1962}, showing that
it gives again a Markov process, the so-called {\it Bochner subordinated Markov
process}. A more interesting scenario is realized when analyzing the case of inverse
subordinators. Indeed, after the time change, we fail to obtain a Markov
process. Therefore, the study of such kind of processes becomes really challenging. 
From this perspective, let us recall the  work
by Montroll and Weiss, \cite{Montroll1965}, where the authors
consider the physically motivated case of random walks in random time.
This seminal paper originated a wide research activity related to the study of Markov
processes with inverse stable subordinators as random times changes,
see the book \cite{Meerschaert2012} for a detailed review and historical
comments.

It is worth mentioning that when we take into account the case 
of processes with random time change which are not subordinators
or inverse subordinators, we can only rely on few results, the overall
field having been less investigated.

Indeed, on the one hand, additional assumptions on the stable subordinator 
considerably reduce the set of time change processes we can count on,
resulting in restrictive assumptions for possible applications. 
On the other hand, we find technical difficulties in
handling general inverse subordinators. Such limitations can be
overcome for certain sub-classes of inverse subordinators, see, e.g.,
\cite{KKS2018,KKdS19}.
Let us underline that the random time change approach turns to be a very effective
tools in modeling several physical systems, spanning from ecological to biological ones, see, e.g., 
\cite{Magdziarz2015} and references therein, also in view of additional applications.

There is a natural question concerning the use of a random
time change not only in stochastic dynamics but also with respect to a wider class
of dynamical problems. In this paper we focus  on the analysis
of latter task in the case of dynamical systems taking values in $\X$.
In particular, let $X(t,x)$, $t\ge0$, $x\in\mathbb{R}^{d}$ be a dynamical system
in $\mathbb{R}^{d}$, starting from $x$ at initial time, namely: $X(0,x)=x$. 
Of course, such a  system is also a deterministic Markov process. Given $f:\mathbb{R}^{d}\longrightarrow\mathbb{R}$
we define 
\[
u(t,x):=f(X(t,x))\,,
\]
hence obtaining a version of the Kolmogorov equation, called the
{\it Liouville equation} within the theory of dynamical systems: 
\[
\frac{\partial}{\partial t}u(t,x)=Lu(t,x)\,,
\]
$L$ being the generator of a semigroup which results to be
the solution of the Liouville equation, see, e.g., \cite{EN2000,RS75,Yosida80}
for more details. 

If $E(t)$ is an inverse subordinator process (see Section \ref{sec:RT-FA} below for details and examples),
then we may consider the time changed random dynamical systems 
\[
Y(t):=X(E(t))\,.
\]
Our aim is to analyze the properties of $Y(t)$ depending on those
of the initial dynamical systems $X(t)$. In particular, we can
define 
\[
v(t,x):=\E[f(Y(t,x)]\,,
\]
then trying to compare the behavior $u(t,x)$ and $v(t,x)$ for a certain
class of functions $f$.

In what follows, we  present the main problems which naturally
appear studying random time changes in dynamical
systems. Moreover, we provide solutions to these problems with respect to 
the examples collected is Section \ref{sec:RT-FA}. 

The rest of the paper is organized as follows. In Section \ref{sec:RT-FA}
we present the classes  of inverse subordinators and the associated general
fractional derivatives. In Section \ref{randomds} we  study random time dynamical systems, also
considering the simplest examples of them. Moreover, we also provide the first results
when the random time is associated to the $\alpha$-stable subordinator.
In Subsection \ref{sec:PGM} we consider a dynamical system as a deterministic
Markov processes, also introducing the notion of potential and Green measure
of the dynamical system. In Subsection \ref{sec:PT} we investigate
the path transformation of a simple dynamical system by a random time.
The last part of the work, namely  Section \ref{sec:RTTE},  is devoted to the analysis of transport equations
and random time changes in this important  class of equations.

		\section{Random times}
	\label{sec:RT-FA}
	In what follows, we recall some preliminary definitions and results related to 
	random times processes and subordinators.
	Let us start with the following fundamental definition
	\begin{definition}
	Let $(\Omega,\mathcal{F},\P)$ be a probability space.  A \emph{random time} is a process $E:[0,+\infty)\times\Omega\to\R^+$ such that
	\begin{itemize}
		\item[(i)]  for $a.e.$ $\omega\in\Omega$ $E(t,\omega)\ge0$ for all $t\in[0,+\infty)$,
		\item[(ii)] for $a.e.$ $\omega\in\Omega$  $E(0,\omega)=0$,
		\item[(iii)]the function $E(\cdot,\omega)$ is increasing and satisfies
		$$
		\lim\limits_{t\to +\infty}E(t,\omega)=+\infty\,.
		$$
	\end{itemize}
\end{definition}
	Concerning the concept of {\it subordinators}, we can introduce it as follows:
	\begin{definition}
		Let $(\Omega,\mathcal{F},\P)$ be a probability space. A process $\{S(t),\, t\ge0\}$ is a \emph{subordinator} if the following conditions are satisfied
		\begin{itemize}
			\item [(i)] $S(0)=0$; 
			\item[(ii)]$S(t+r)-S(t)$ has the same law of $S(r)$ for all $t,r>0\,$;
			\item [(iii)] if ${(\mathcal{F}_t)}_t$ denotes the filtration generated by ${(S(t))}_t$, i.e. $\mathcal{F}_t=\sigma(\{S(r), r\le t\})$, then $S(t+r)-S(t)$ is independent of $\mathcal{F}_t$ for all $t,r>0\,;$
			\item [(iv)] $t\to S(t)(\omega)$ is almost surely right-continuous with left limits;
			\item [(v)] $t\to S(t)$ is almost surely an increasing function.
		\end{itemize}
	\end{definition}
	\noindent
	For the sake of completeness, let us note that the process $S(\cdot)$ is a \emph{L\'{e}vy process} if it satisfies the conditions $(i)-(iv)$, see, e.g., \cite{Bertoin96} for more details. 
	Let $S=\{S(t),\;t\ge0\}$ be  L{\'e}vy process, then  its Laplace transform can be written in terms 
	of a Bernstein function (also known as Laplace exponent) $\Phi:[0,\infty)\longrightarrow[0,\infty)$ by 
	\[
	\mathbb{E}[e^{-\lambda S(t)}]=e^{-t\Phi(\lambda)},\quad\lambda\ge0\,.
	\]
	Moreover, the function $\Phi$ admits the representation 
	\begin{equation}
		\Phi(\lambda)=\int_{0}^{+\infty}(1-e^{-\lambda\tau})\,\mathrm{d}\sigma(\tau),\label{eq:Levy-Khintchine}
	\end{equation}
	where the measure $\sigma$, also called L{\'e}vy measure, has support
	in $[0,\infty)$ and fulfills 
	\begin{equation}
		\int_{0}^{+\infty}(1\wedge\tau)\,\mathrm{d}\sigma(\tau)<\infty\label{eq:Levy-condition}\,.
	\end{equation}
	Let $\sigma$ be a L{\'e}vy measure, we define the associated kernel $k$ as
	follows 
	\begin{align}
		k:(0,\infty)&\longrightarrow(0,\infty),\;\label{eq:k}\\
		t&\mapsto k(t):=\sigma\big((t,\infty)\big)\,.\nonumber
	\end{align}	
	Its Laplace transform is denoted by $\mathcal{K}$, and, for any	$\lambda\ge0$, one has 
	\begin{equation}
		\mathcal{K}(\lambda):=\int_{0}^{\infty}e^{-\lambda t}k(t)\,\mathrm{d}t\label{eq:LT-k}\,.
	\end{equation}
	We note that the relation between the function $\mathcal{K}$ and the Laplace exponent
	$\Phi$ is given by 
	\begin{equation}
		\Phi(\lambda)=\lambda\mathcal{K}(\lambda),\quad\forall\lambda\ge0\label{eq:Laplace-exponent}\,.
	\end{equation}
	Throughout the paper we shall suppose that 
	\begin{hp}\label{hp}
		\label{cond} Le $\Phi$ be a complete Bernstein function, that is, the L{\'e}vy
		measure $\sigma$ is absolutely continuous with respect to the Lebesgue
		measure. The functions $\mathcal{K}$ and $\Phi$ satisfy 
		\begin{equation*}
			\mathcal{K}(\lambda)\to\infty,\text{ as \ensuremath{\lambda\to0}};\quad\mathcal{K}(\lambda)\to0,\text{ as \ensuremath{\lambda\to\infty}};\label{eq:H1}
		\end{equation*}
		\begin{equation*}
			\Phi(\lambda)\to0,\text{ as \ensuremath{\lambda\to0}};\quad\Phi(\lambda)\to\infty,\text{ as \ensuremath{\lambda\to\infty}}.\label{eq:H2}
		\end{equation*}
	\end{hp}
	\begin{example}[$\alpha$-stable subordinator]
		\label{exa:alpha-stable1}A classical example of a subordinator $S$
		is the so-called $\alpha$-stable process with index $\alpha\in(0,1)$.
		In particular,
		a subordinator is $\alpha$-stable if its Laplace exponent
		is 
		\[
		\Phi(\lambda)=\lambda^{\alpha}=\frac{\alpha}{\Gamma(1-\alpha)}\int_{0}^{\infty}(1-e^{-\lambda\tau})\tau^{-1-\alpha}\,\mathrm{d}\tau\,,
		\]
		where $\Gamma$ is the gamma function.\\
		In this case, the associated  L{\'e}vy measure is given by $\mathrm{d}\sigma_{\alpha}(\tau)=\frac{\alpha}{\Gamma(1-\alpha)}\tau^{-(1+\alpha)}\,\mathrm{d}\tau$ and
		the corresponding kernel $k_{\alpha}$ has the form $$k_{\alpha}(t)=g_{1-\alpha}(t):=\frac{t^{-\alpha}}{\Gamma(1-\alpha)} \ \ \ \ \ t\ge0\,,$$
		with Laplace transform equal to $\mathcal{K}_{\alpha}(\lambda)=\lambda^{\alpha-1}$, for
		$\lambda\ge0$. 
	\end{example}
	
	\begin{example}[Gamma subordinator]
		\label{exa:gamma-subordinator}The Gamma process $Y^{(a,b)}$ with
		parameters $a,b>0$ is another example of a subordinator with Laplace
		exponent 
		\[
		\Phi_{(a,b)}(\lambda)=a\log\left(1+\frac{\lambda}{b}\right)=\int_{0}^{\infty}(1-e^{-\lambda\tau})a\tau^{-1}e^{-b\tau}\,\mathrm{d}\tau\,, 	\]
		the second equality being the Frullani integral. The associated L{\'e}vy
		measure is given by $d\sigma_{(a,b)}(\tau)=a\tau^{-1}e^{-b\tau}\,\mathrm{d}\tau$,
		with associated kernel equal to
		$$k_{(a,b)}(t)=a\Gamma(0,bt), \ \ t>0\,,
		$$
		where $$
		\Gamma(\nu,z):=\int_{z}^{\infty}e^{-t}t^{\nu-1}\,\mathrm{d}t$$
		is the incomplete gamma function, see, e.g., \cite[Section 8.3]{GR81} for more details. 
		Moreover, its 	Laplace transform is $$
		\mathcal{K}_{(a,b)}(\lambda)=a\lambda^{-1}\log\left(1+\frac{\lambda}{b}\right),  \ \  \ \ \ \lambda>0\,.$$
	\end{example}
	
	\begin{example}[Truncated $\alpha$-stable subordinator]
		\label{exa:truncated_stable_ sub}The truncated $\alpha$-stable
		subordinator, see \cite[Example 2.1-(ii)]{Chen2017}, $S_{\delta}$,
		$\delta>0$, constitutes an example of a driftless $\alpha$-stable subordinator with L{\'e}vy
		measure given by 
		\[
		\mathrm{d}\sigma_{\delta}(\tau):=\frac{\alpha}{\Gamma(1-\alpha)}\tau^{-(1+\alpha)}1\!\!1_{(0,\delta]}(\tau)\,\mathrm{d}\tau,\qquad\delta>0\,.
		\]
		The corresponding Laplace exponent is given by
		\[
		\Phi_{\delta}(\lambda)=\lambda^{\alpha}\left(1-\frac{\Gamma(-\alpha,\delta\lambda)}{\Gamma(-\alpha)}\right)+\frac{\delta^{-\alpha}}{\Gamma(1-\alpha)}\,,
		\]
		with associated kernel 
		\[
		k_{\delta}(t):=\sigma_{\delta}\big((t,\infty)\big)=\frac{1\!\!1_{(0,\delta]}(t)}{\Gamma(1-\beta)}(t^{-\beta}-\delta^{-\beta}),\;t>0\,.
		\]
	\end{example}
	
	\begin{example}[Sum of two alpha stable subordinators]
		\label{exa:sum-two-stables}Let $0<\alpha<\beta<1$ be given and let
		$S_{\alpha,\beta}(t)$, for $t\ge0$, be the driftless subordinator with Laplace
		exponent given by 
		\[
		\Phi_{\alpha,\beta}(\lambda)=\lambda^{\alpha}+\lambda^{\beta}\,.
		\]
		Then, by Example \ref{exa:alpha-stable1}, we have   that the corresponding L{\'e}vy measure
		$\sigma_{\alpha,\beta}$ is the sum of two L{\'e}vy measures. Indeed, it holds
		\[
		\mathrm{d}\sigma_{\alpha,\beta}(\tau)=\mathrm{d}\sigma_{\alpha}(\tau)+\mathrm{d}\sigma_{\alpha}(\tau)=\frac{\alpha}{\Gamma(1-\alpha)}\tau^{-(1+\alpha)}\,\mathrm{d}\tau+\frac{\beta}{\Gamma(1-\beta)}\tau^{-(1+\beta)}\,\mathrm{d}\tau\,,
		\]
		implying that  the associated kernel $k_{\alpha,\beta}$ reads as follow 
		\[
		k_{\alpha,\beta}(t):=g_{1-\alpha}(t)+g_{1-\beta}(t)=\frac{t^{-\alpha}}{\Gamma(1-\alpha)}+\frac{t^{-\beta}}{\Gamma(1-\beta)},\;t>0\,,
		\]
		with associated Laplace transform given by 
		$$\mathcal{K}_{\alpha,\beta}(\lambda)=\mathcal{K}_{\alpha}(\lambda)+\mathcal{K}_{\beta}(\lambda)=\lambda^{\alpha-1}+\lambda^{\beta-1}\,,\,
\ \ \		\lambda>0\,.$$ 
	\end{example}
	
	\begin{example}[Kernel with exponential weight]
		\label{exa:exponential-weight} Taking
		$\gamma>0$ and $0<\alpha<1$, let us consider the subordinator with Laplace
		exponent 
		\[
		\Phi_{\gamma}(\lambda):=(\lambda+\gamma)^{\alpha}=\left(\frac{\lambda+\gamma}{\lambda}\right)^{\alpha}\frac{\alpha}{\Gamma(1-\alpha)}\int_{0}^{\infty}(1-e^{-\lambda\tau})\tau^{-1-\alpha}\,\mathrm{d}\tau\,.
		\]
		Then the associated 
		L{\'e}vy measure is given by 
		$$\mathrm{d}\sigma_{\gamma}(\tau)=\left(\frac{\lambda+\gamma}{\lambda}\right)^{\alpha}\frac{\alpha}{\Gamma(1-\alpha)}\tau^{-(1+\alpha)}\mathrm{d}\tau\,,$$
		which implies a kernel $k_{\gamma}$ with exponential weight. In particular, we have
		\[
		k_{\gamma}(t)=g_{1-\alpha}(t)e^{-\gamma t}=\frac{t^{-\alpha}}{\Gamma(1-\alpha)}e^{-\gamma t}\,.
		\]
		The corresponding Laplace transform of $k_{\gamma}$ is then given by $\mathcal{K}_{\gamma}(\lambda)=\lambda^{-1}(\lambda+\gamma)^{\alpha}$,
		$\lambda>0$. 
	\end{example}

\subsection{Inverse subordinators and general fractional derivatives}\label{inversub}
In this section we introduce the inverse subordinators and the corresponding general fractional derivatives.
\begin{definition}
	Let $S(\cdot)$ be a subordinator. We define $E(\cdot)$ as the inverse process  of $S(\cdot)$, i.e.
	$$
	E(t):=\inf\left\{ r>0\, |\, S(r)> t \right\} =\sup\left\{ r\ge0\, |\, S( r)\le  t \right\}\, \ \ \mbox{for all} \ t \in[0,+\infty)\,.
	$$
\end{definition}
\noindent
For any $t\ge0$, we denote by $G_{t}^{k}(\tau):=G_{t}(\tau)$, $\tau\ge0$
the marginal density of $E(t)$ or, equivalently 
\begin{equation*}
	G_{t}(\tau)\,\mathrm{d}\tau=\frac{\partial}{\partial\tau}\P\xz(E(t)\le\tau)\,\mathrm{d}\tau=\frac{\partial}{\partial\tau}\P\xz(S(\tau)\ge t)\,\mathrm{d}\tau=-\frac{\partial}{\partial\tau}\P\xz(S(\tau)<t)\,\mathrm{d}\tau.
\end{equation*}
\begin{remark}
	\label{rem:distr-alphastab-E}If $S$ is the $\alpha$-stable process,
	$\alpha\in(0,1)$, then the inverse process $E(t)$ has Laplace transform, see \cite[Prop.~1(a)]{Bingham1971}, given by 
	\begin{equation}
		\mathbb{E}[e^{-\lambda E(t)}]=\int_{0}^{\infty}e^{-\lambda\tau}G_{t}(\tau)\,\mathrm{d}\tau=\sum_{n=0}^{\infty}\frac{(-\lambda t^{\alpha})^{n}}{\Gamma(n\alpha+1)}=E_{\alpha}(-\lambda t^{\alpha})\label{eq:Laplace-density-alpha} \,.
	\end{equation}
	By the asymptotic behavior of the Mittag-Leffler function
	$E_{\alpha}$, it follows that $\mathbb{E}[e^{-\lambda E(t)}]\sim Ct^{-\alpha}$
	as $t\to\infty$. Using the properties of the Mittag-Leffler function
	$E_{\alpha}$, we can show that the density $G_{t}(\tau)$ is given
	in terms of the Wright function $W_{\mu,\nu}$, namely $G_{t}(\tau)=t^{-\alpha}W_{-\alpha,1-\alpha}(\tau t^{-\alpha})$,
	see \cite{Gorenflo1999} for more details. 
\end{remark}
For a general subordinator, the following lemma determines the $t$-Laplace
transform of $G_{t}(\tau)$, with $k$ and $\mathcal{K}$ given in
\eqref{eq:k} and \eqref{eq:LT-k}, respectively. 
\begin{lemma}
	\label{lem:t-LT-G}The $t$-Laplace transform of the density $G_{t}(\tau)$
	is given by 
	\begin{equation}
		\int_{0}^{\infty}e^{-\lambda t}G_{t}(\tau)\,\mathrm{d}t=\mathcal{K}(\lambda)e^{-\tau\lambda\mathcal{K}(\lambda)}.\label{eq:LT-G-t}
	\end{equation}
	The double ($\tau,t$)-Laplace transform of $G_{t}(\tau)$ is 
	\begin{equation}
		\int_{0}^{\infty}\int_{0}^{\infty}e^{-p\tau}e^{-\lambda t}G_{t}(\tau)\,\mathrm{d}t\,\mathrm{d}\tau=\frac{\mathcal{K}(\lambda)}{\lambda\mathcal{K}(\lambda)+p}\,.\label{eq:double-Laplace}
	\end{equation}
\end{lemma}
\begin{proof}
	For the proof see \cite{Kochubei11} or  \cite[Lemma 3.1]{Toaldo2015}
\end{proof}
Let us now recall the definition of {\it General Fractional Derivative} (GFD) associated to a kernel $k$, see \cite{Kochubei11} and references therein for more details.
\begin{definition}
	Let $S$ be a subordinator and the kernel
	$k\in L_{\mathrm{loc}}^{1}(\mathbb{R}_{+})$ given in \eqref{eq:k}.
	We define a differential-convolution operator by 
	\begin{equation}
		\big(\mathbb{D}_{t}^{(k)}u\big)(t)=\frac{d}{dt}\int_{0}^{t}k(t-\tau)u(\tau)\,\mathrm{d}\tau-k(t)u(0),\;t>0\,.\label{eq:general-derivative}
	\end{equation}
\end{definition}
\begin{remark}
	The operator $\mathbb{D}_{t}^{(k)}$ is also known as \emph{Generalized
		Fractional Derivative.}
\end{remark}
\begin{example}[Distributed order derivative]
	\label{exa:distr-order-deriv}Consider the kernel $k$ defined by
	\begin{equation}
		k(t):=\int_{0}^{1}g_{\alpha}(t)\,\mathrm{d}\alpha=\int_{0}^{1}\frac{t^{\alpha-1}}{\Gamma(\alpha)}\,\mathrm{d}\alpha,\quad t>0\,.\label{eq:distributed-kernel}
	\end{equation}
	Then it is easy to see that 
	\[
	\mathcal{K}(\lambda)=\int_{0}^{\infty}e^{-\lambda t}k(t)\,\mathrm{d}t=\frac{\lambda-1}{\lambda\log(\lambda)},\quad\lambda>0\,.
	\]
	The corresponding differential-convolution operator $\mathbb{D}_{t}^{(k)}$
	is called distributed order derivative, see, e.g., \cite{Atanackovic2009,Daftardar-Gejji2008,Gorenflo2005,Hanyga2007,Kochubei2008,Meerschaert2006}
	for more details and applications. 
\end{example}
\noindent
We conclude this section with a result that will be useful later on, starting by recalling
the following definition.
\begin{definition}
	Given the functions $f:\R\to\R$ and $g:\R\to\R$, we say that $f$ and $g$ are \emph{asymptotically equivalent
	at infinity}, and denote $f\sim g$ as $x\to+\infty$, if
\[
\lim_{x\to+\infty}\frac{f(x)}{g(x)}=1\,.
\]
Moreover, we say that $f$ is slowly varying 
if 
\[
\lim_{x\to+\infty}\frac{f(\lambda x)}{f(x)}=1,\quad\mbox{for any }\lambda>0\,.
\]
\end{definition}
\noindent
For more details on slowly varying functions,  we refer the interested reader to, e.g., \cite{Feller71,Seneta1976}.\\
\begin{lemma}
	\label{KL} Suppose hypotheses \ref{hp} are satisfied, and that the subordinator $S(t)$, along with its inverse $E(t)$,
	$t\ge0$ , are such that 
	\begin{equation}
		\mathcal{K}(\lambda)\sim\lambda^{-\gamma}Q\left(\frac{1}{\lambda}\right),\quad\lambda\to0\label{eq:condition-K}\,,
	\end{equation}
	where $0\leq\gamma\leq1$ and $Q(\cdot)$ is a slowly varying function. Moreover, define 
	\[
	A(t,z):=\int_{0}^{\infty}e^{-z\tau}G_{t}(\tau)\,\mathrm{d}\tau,\quad t>0,\;z>0\,.
	\]
	Then it holds
	\[
	A(t,z)\sim\frac{1}{z}\frac{t^{\gamma-1}}{\Gamma(\gamma)}Q(t),\quad t\to\infty\,.
	\]
\end{lemma}
\begin{proof}
	For the proof see  \cite[Theorem 4.3]{KocKon2017}.
\end{proof}
\begin{remark}
	We point out that the condition \eqref{eq:condition-K} on the Laplace
	transform of the kernel $k$ is satisfied by all Examples \ref{exa:alpha-stable1}--\ref{exa:exponential-weight}
	and \ref{exa:distr-order-deriv}, stated above. The case of Example \ref{exa:sum-two-stables}
	is easily checked as 
	\[
	\mathcal{K}(\lambda)=\lambda^{\alpha}+\lambda^{\beta}=\lambda^{-(1-\alpha)}(1+\lambda^{-(\alpha-\beta)})=\lambda^{-\gamma}Q\left(\frac{1}{\lambda}\right)\,,
	\]
	where $\gamma=1-\alpha>0$ and $Q(t)=1+t^{\alpha-\beta}$ is a slowly
	varying function. 
	
\end{remark}

	\subsection{Compound Poisson Process}
	A significant example of subordinator is given by the \emph{Compound Poisson Process} (CPP). Roughly speaking, a CPP is a jump (stochastic) process, whose both  jumps size  and the number of them, are independent random variables.\\
	First, we define a random process $N(t)$ modeling the number of jumps that occurred in given time interval $[0,t]$, $t>0$.
	
	\begin{definition}
		Let  $(\Omega,\mathcal{F},\P)$  be a probability space.	A random process $N:[0,+\infty)\times\Omega\to\N\cup\{0\}$ is a {discrete Poisson process} with rate $\lambda>0$ if it satisfies the following properties
		\begin{itemize}
			\item[(i)]  $N(0)=0\,;$
			\item[(ii)] $\forall t,s\in [0,+\infty)$ such that $t>s$ one has that $N(t)-N(s)$ is independent of $N(s)\,;$
			\item [(iii)] the random variable $N(t)-N(s)$ has a Poisson distribution with parameter $\lambda(t-s)\,.$
		\end{itemize}
	\end{definition}
	\begin{remark}
		The property (iii) implies that $N(\cdot,\omega)$ is increasing for {\it almost all} $\omega\in\Omega$, namely
		$$\P\{\omega \in \Omega: N(\cdot,\omega) {\text{ is not increasing}} \}=0\,.$$
		Moreover
		$$
		\lim\limits_{t\to+\infty}N(t,\omega)=+\infty\,,\,\mbox{for a.a. }\omega\in\Omega\,.
		$$
	\end{remark}
	A CPP is defined as follows.
	\begin{definition}
		Let $N(\cdot)$ be a discrete Poisson process with rate $\lambda$, then $S(\cdot)$ is said to be a CPP of rate $\lambda>0$ if it admits the following representation
		\begin{equation*}
			S(t)=\sum\limits_{i=1}^{N(t)} R_i\,,
		\end{equation*}
		where $\{R_i\}_i$ are non-zero and non-negative i.i.d. random variables independent of $N(\cdot)\,$.
	\end{definition}
	\noindent
	It is straightforward to note that a CPP is also a random time process.
	\begin{remark}
		The random variables $R_i$ represent the \emph{jumps} of the process $S$, while $N(t)$ is the number of jumps occurred in $[0,t]$. Moreover, for each $\omega\in\Omega$, $N(\cdot,\omega)$ is represented by an increasing step function.
	\end{remark}
	It is well known that the moment generating function (or Laplace transform) of a CPP of parameter $\lambda$ is given by
	$$
	L_{S(t)}(s):=\E[e^{sS(t)}]=e^{\lambda t(L_R(s)-1)}\,,
	$$
	 $L_R(\cdot)$ being  the moment generating function of the random variables $R_i\,.$
	This results holds true for all $s$ in the domain of $L_R$.\\
	According to the definitions given in \eqref{eq:Levy-Khintchine}-\eqref{eq:Laplace-exponent}, for a CPP we have:
	\begin{itemize}
		\item the Laplace exponent is given by
		$$
		\Phi(s)=\lambda(1-L_R(-s))=\lambda\E[1-e^{-sR}]=\lambda\int_0^{+\infty}(1-e^{-sr})\,\mathrm{d}F_R(r)\,,
		$$
		where $F_R(\cdot)$ is the cumulative distribution function of $R$;
		\item the associated L\'evy measure and kernel are respectively given by
		$$
		\sigma((a,b))=\lambda\,\P(a\le R\le b)\;,\; k(t)=\sigma((t,+\infty))=\lambda\,\P(R\ge t)\,,
		$$
		while the Laplace transform of $k$ reads as follow
		$$
		\mathcal{K}(s)=\frac{\Phi(s)}s=\frac{\lambda(1-L_R(-s))}{s}\,.
		$$
	\end{itemize}
	
	Let us note that in this latter case Hypotheses \ref{hp} are not satisfied, since $\Phi(s)\to \lambda$ when $s\to+\infty\,.$
	\section{Random time dynamical systems}\label{randomds}

	\subsection{Dynamical systems and Liouville equations}	\label{sec:RTD}
	There is a natural question concerning the use of a random time change not only in stochastic dynamics, but more generally in an ample
	class of dynamical problems. In what follows, we shall focus the attention on the analysis of the {\it random time change approach}
		for dynamical systems taking values in $\R^d$.\\
Let $X(t,x)$, $t \geq0$ be a dynamical system in $\R^d$ such that $X(0, x) = x \in \R^d$. Such a system is also a deterministic Markov process. Therefore, given $f : \R^d \rightarrow \R$, and defining
	$$
	u(t,x) := f(X(t,x))\,,
	$$
	we have a version of the Kolmogorov equation, which is nothing but the Liouville equation within the theory of dynamical systems. Indeed, 
	\begin{equation}\label{uode}
	u_t(t,x)= Lu(t,x)\,,
	\end{equation}
	where $L$ is the generator of the semigroup solution of the Liouville equation, see, e.g., \cite{EN2000,RS75,Yosida80}. 
	\subsection{Random time changes and fractional Liouville equations}
Let  $(\Omega,\mathcal{F},\P)$ be a probability space. Let $X(t,x)$, $t \geq0$, be a dynamical system in $\R^d$ starting at time $t=0$ from $x\in\R^d$. Given an inverse subordinator process $E(\cdot)$,
	we consider the time changed random dynamical systems
	$$
	Y(t,\omega;x)=X(E(t,\omega);x)\,,\qquad t\in [0,+\infty), \ x \in\X , \  \omega \in \Omega\,.
	$$
	For a suitable $f:\R^d\rightarrow \R$ we define
	\begin{equation}
		v(t,x):=\E[f(Y(t;x))]\,,
	\end{equation}
	where, without loss of generality, with $E(t)$ and $Y(t;x)$ we shortly refer to $E(t,\cdot)$, resp. to $Y(t,\cdot \ ;x).$	\\
	As pointed out in, e.g., \cite{Chen2017,Toaldo2015}, $v(t,x)$ solves an evolution equation with the generator $L$,  with generalized fractional derivative (see \eqref{eq:general-derivative}), i.e.
	\begin{equation}
		\label{fde-for-sub-sol}
		\mathbb{D}_{t}^{(k)}v( \cdot,x)(t)=Lv(t,x)\,.
	\end{equation}
Let $u(t,x)$ be the solution to \eqref{uode} with the same generator $L$ in \eqref{fde-for-sub-sol}. Under quite general assumptions there is an essentially obvious relation between these evolutions
	\begin{equation}
		v(t,x)=\int_{0}^{\infty}u(\tau,x)G_{t}(\tau)\,\mathrm{d}\tau,\label{SUBFOR}
	\end{equation}
$G_{t}(\tau)$ being the density of $E(t)$, as defined in Section \ref{inversub}.

Having in mind the analysis of the random time change influence on the asymptotic properties of $v(t, x)$, we may suppose that the latter formula gives all necessary technical equipments. Unfortunately, the situation is essentially more complicated. In fact, the knowledge we have of the properties characterizing the density $G_t(\tau)$ is, in general, very poor.  
The aim of this section is to describe a class of subordinators for which we may obtain information about the time asymptotic of the generalized fractional dynamics.
	
	\subsection{First examples}\label{fex}
	We consider
	the simplest evolution equation in $\X$
	\[
	{\mathrm{d}}X(t)=v\mathrm{d}t\in\X,\quad X(0)=x_{0}\in\X\,,
	\]
	with corresponding dynamics given by 
	\[
	X(t)=x_{0}+vt,\quad t\geq0\,.
	\]
	Without loss of generality, let us assume that $x_{0}=0$. Then, we take $f(x)=e^{-\alpha|x|},\;\alpha>0$.
	Hence, the corresponding solution to the Liouville equation is 
	\[
	u(t,x)=e^{-\alpha t|v|},\quad t\geq0\,.
	\]
	
	\begin{proposition}
		Assume that the assumptions
		of Lemma \ref{KL} are satisfied. Then
		\[
		v(t,x)\sim\frac{1}{\alpha|v|\Gamma(\gamma)}t^{\gamma-1}Q(t),\quad t\to\infty\,.
		\]
		
	\end{proposition}
	
	\begin{proof}
		
		From the explicit form of the solution $u(t,x)$, and
		using both (\ref{SUBFOR}) and Lemma \ref{KL}, we obtain 
		\[
		v(t,x)\sim\frac{1}{\alpha|v|\Gamma(\gamma)}t^{\gamma-1}Q(t),\quad t\to\infty\,.
		\]
		In particular, for the $\alpha$-stable subordinator considered in
		Example \ref{exa:alpha-stable1}, we obtain $v(t,x)\sim Ct^{-\alpha}$,
		for a given constant $C>0$. Therefore, starting with a solution $u(t,x)$
		with exponential decay after subordination, we observe a polynomial
		decay with the order defined by the random time characteristics. 
	\end{proof}
\noindent
	For $d=1$ consider the dynamics
	\[
	\beta{\mathrm{d}}X(t)=\frac{1}{X^{\beta-1}(t)}\mathrm{d}t,\quad\beta\geq1\,,
	\]
	then the solution is given by 
	\[
	X(t)=(t+C)^{1/\beta}\,.
	\]
	Considering the function $f(x)=\exp(-a|x|^{\beta})$, $a>0$, and supposing that the assumptions of Lemma \ref{KL} are satisfied, then, exploiting the explicit form of the solution $u(t,x)$, we have that the long time behavior of the subordination $v(t,x)$ is given by 
	\[
	v(t,x)\sim\frac{e^{-aC}}{a}\frac{t^{\gamma-1}}{\Gamma(\gamma)}Q(t),\quad t\to\infty\,.
	\]
	In particular, choosing the density $G_{t}(\tau)$ of the inverse
	subordinator $E(t)$ as in the Example \ref{exa:sum-two-stables},
	we obtain
	\[
	v(t,x)\sim Ct^{-\alpha}(1+t^{\alpha-\beta})\sim Ct^{-\alpha},\quad t\to\infty\,.
	\]

	\subsection{Green measures} 	\label{sec:PGM}

The notion of potential is a classical topic within  the theory
of Markov processes, see, e.g., \cite{Blumenthal1968}. Recently, it
has been  proposed the concept of Green measure as a representation of potentials in an integral form, see \cite{KdS2020}. The modification of these
concepts for time changed Markov processes was investigated in \cite{KdS20}. Considering a dynamical system as a deterministic Markov processes,
we have the possibility to study the notion of potential and Green measure in this context.\\
According with the above stated framework, given a function $f:\X\to\R$, we consider the solution
to the Cauchy problem 
\begin{align*}
	\begin{cases}
		u_t(t,x)\!\!\!&=Lu(t,x),\\
		u(0,x)\!\!\!&=f(x),
	\end{cases}
\end{align*}
obtaining  
\[
u(t,x)=(e^{tL}f)(x)\,.
\]
Then, by defining the potential for the function $f$ as 
\[
U(f,x):=\int_{0}^{\infty}u(t,x)\,\mathrm{d}t=\int_{0}^{\infty}(e^{tL}f)(x)\,\mathrm{d}t=-(L^{-1}f)(x), \quad x\in\X\,,
\]
the existence of $U(f,x)$ is not clear at all. Indeed, it depends on the
class of functions $f$ and the Liouville generator $L$. Assuming
the existence of $U(f,x)$ we aim at obtaining an integral representation 
\begin{equation}
	\label{eq:potential-U}
	U(f,x)=\int_{\X}f(y)\,\mathrm{d}\mu^{x}(y)\,,
\end{equation}
$\mu^{x}$ being a random  measure on $\X$, that we will call
the Green measure of our dynamical system. As in the case of Markov
processes, the definition of the potential is easy to introduce but
difficult to analyse for each particular model. Moreover,  on the base of specific  examples, we 
may assume that the potentials are well  defined for
special classes  of functions $f$. Nevertheless, we can not expect  the existence of an associated Green
measure. 

As already seen, after a random time change we will have the subordinated solution
$v(t,x)$ for the fractional equation, see equation \eqref{fde-for-sub-sol}. 
Then we can try  to {\it re}-define the potential 
\[
V(f,x):=\int_{0}^{\infty}v(t,x)\,\mathrm{d}t, \quad x\in\X\,,
\]
which turns to be divergent for general random times.
Indeed, by the subordination formula \eqref{SUBFOR}
and the Fubini theorem, we have
\[
V(f,x)=\int_{0}^{\infty}\int_{0}^{\infty}u(\tau,x)G_{t}(\tau)\,\mathrm{d}\tau\,\mathrm{d}t
=\int_{0}^{\infty}u(\tau,x)\left(\int_{0}^{\infty}G_{t}(\tau)\,\mathrm{d}t\right)\mathrm{d}t\,\mathrm{d}\tau\,,
\]
where the inner integral is not convergent because of the equality \eqref{eq:LT-G-t} together with the Hypotheses~\ref{eq:H1}.
To overcome  this difficulty we may use the notion of renormalized potential. 
In particular,  inspired by the time change of Markov processes (see \cite{KdS20} for details), we define
the renormalized potential 
\begin{equation}
	V_r(f,x):=\lim_{t\to\infty}\frac{1}{N(t)}\int_0^tv(s,x)\,\mathrm{d}s, \quad t\ge0\,,
\end{equation}
where  $N(t)$ is defined by $N(t):=\int_0^tk(s)\,\mathrm{d}s$.
Then by assuming the existence of $U(f,x)$, we have
\[
V_r(f,x)=\int_{0}^{\infty}u(t,x)\,\mathrm{d}t\,.
\]

	\subsection{Path transformations}
	
	\label{sec:PT} Let us now investigate how the trajectories
	of dynamical systems transform under random times. 
	According to what seen above, we consider the Liouville
	equation for 
	\[
	u(t,x):=f(X(t,x)),\quad t\ge0,\;x\in\X\,,
	\]
	that is, 
	\[
	u_t(t,x)=Lu(t,x),\quad u(0,x)=f(x)\,,
	\]
	$L$ being the generator of a semigroup. In addition, let $E(t)$,
	$t\ge0$, be the inverse subordinator process. Then we can consider
	the time changed random dynamical systems 
	\[
	Y(t,x)=X(E(t),x),\quad t\ge0,\;x\in\X\,,
	\]
	where, without loss of generality, $E(t)$, resp. $Y(t;x)$, shortly refer to $E(t,\cdot)$, resp. to $Y(t,\cdot \ ;x)\,.$
	Definining 
	\[
	v(t,x):=\E[f(Y(t,x)]\,,
	\]
	by the subordination formula, we have 
	\[
	v(t,x)=\int_{0}^{\infty}u(\tau,x)G_{t}(\tau)\,\mathrm{d}\tau\,.
	\]
	Considering the vector-function $f:\R^d\rightarrow \R$ defined as
	\[
	f(x)=x\,,
	\]
	we have that the average trajectories of $Y(t,x)$ read as follow
	\[
	\E[Y(t,x)]= \int_{0}^{\infty}X(\tau,x)G_{t}(\tau)\,\mathrm{d}\tau\,.
	\]
	Then considering the dynamical system of Section \ref{fex}, namely $X(t,x)=vt$,  we obtain 
	\[
	\mathbb{E}[Y(t,x)]=v\int_{0}^{\infty}\tau  G_{t}(\tau)\,\mathrm{d}\tau\,.
	\]
	Therefore, we need to know the first moment of the density $G_t$. Considering the case of the inverse $\alpha$-stable subordinator stated in Example \ref{exa:alpha-stable1}, we have
	$$
	\int_0^\infty \tau G_t(\tau)\,\mathrm{d}\tau = C t^{\alpha}\,.
	$$
	Therefore, the asymptotic of the time changed trajectory will be slower (proportional to $t^{\alpha}$) instead of initial linear $vt$ motion. In a forthcoming paper we will study in detail these results  for other classes  of inverse subordinators. 
	\section{Random time transport equations} 	\label{sec:RTTE}
	
	Let $b(\cdot):\R^d\to\R^d$, $d\ge 1$, be a bounded continuous vector field. We consider the following dynamical system
	\begin{equation}\label{dyn}
		\begin{cases}
			\mathrm{d}X(t;x)=b(X(t;x))\,\mathrm{d}t\,,\\
			X(0;x)=x\,,
		\end{cases}
	\end{equation}
	with starting point  $x\in\R^d$, at initial time $t=0$. Let us consider a bounded continuous function $f:\R^d\to\R$ and define
	\begin{equation}\label{u}
		u(t,x):=f(X(t;x))\,,
	\end{equation}
	where $X$ is defined by \eqref{dyn}.\\
	In what follows, we prove that $u$ is a classical solution of the  first-order parabolic equation, provided  $b$ and $f$ are regular enough.
	\begin{proposition}\label{r1}
		Let $b$ and $f$ be bounded continuous functions. Then $u$, defined by \eqref{u}, is a classical solution of the following first-order parabolic equation 
		\begin{equation}\label{pde}
			\begin{cases}
				u_t(t,x)=b(x)\cdot \nabla u(t,x)\,,    \forall (t,x) \in(0,+\infty)\times\R^d, \\
				u(0,x)=f(x),\, \forall x\in \R^d\,.
			\end{cases}
		\end{equation}
	\end{proposition}
	\begin{remark}
		In case of non-autonomous drift, i.e. $b=b(t,x)$, the equation \eqref{pde} fails to be true, as shown
		by  the following counter-example. If $d=1$ and
		$$
		f(x)=x\,,\qquad b(t,x)=t+x\,,
		$$
		the solution of \eqref{dyn} is $u(t,x)=X(t;x)=(1+x)e^t-t-1\,.$ Then, one has
		$$
		u_t(t,x)=(1+x)e^t-1\,,\qquad b(t,x) u_x(t,x)=(t+x)e^t\,,
		$$
		and the equation \eqref{pde} is not satisfied.
	\end{remark}

	\begin{proof}[Proof of Proposition \ref{r1}]
		By assumptions on $b$ and $f$, we have
		\begin{align*}
			&u_t(t,x)=\nabla f(X(t;x))\cdot b(X(t;x))\,,\\
			&b(x)\cdot\nabla u(t,x)=b(x)\cdot\left(\nabla f(X(t;x))\mathrm{Jac}_x(X(t;x))  \right)\,.
		\end{align*}
		To prove \eqref{pde}, we have to check
		$$
		b(X(t;x))=\mathrm{Jac}_x(X(t;x))b(x)\,,
		$$
		equivalently
		$$
		\varphi(t,x):=b(X(t;x))-\mathrm{Jac}_x(X(t;x))b(x)=0\qquad\forall\, t\ge0,\,x\in\R^d\,.
		$$
		Computing the time derivative of $\varphi$, we obtain
		$$
		\varphi_t(t,x)=\mathrm{Jac}\,b(X(t;x))b(X(t;x))-\partial_t\mathrm{Jac}_x(X(t;x))b(x)\,.
		$$
		By differentiating equation \eqref{dyn} with respect to $x$, we have
		\begin{equation*}
			\begin{cases}
				\partial_t\mathrm{Jac}_x(X(t;x))=\mathrm{Jac}\,b(X(t;x))\mathrm{Jac}_x(X(t;x))
				\,,\\
				\mathrm{Jac}_x(X(0;x))=I_{d\times d}\,,
			\end{cases}
		\end{equation*}
		and the derivative $\varphi_t$ becomes
		$$
		\varphi_t(t,x)=\mathrm{Jac}\,b(X(t;x))(b(X(t;x))-\mathrm{Jac}_x(X(t;x))b(x))=\mathrm{Jac}\,b(X(t;x))\varphi(t,x)\,.
		$$
		Thus, $\varphi$ satisfies the following \textit{ODE} in time
		\begin{equation}\label{ODE}
			\begin{cases}
				\varphi_t(t,x)=\alpha(t,x)\varphi(t,x)\,,\\
				\varphi(0,x)=0\,,
			\end{cases}
		\end{equation}
		where $\alpha(t,x):=\mathrm{Jac}\,b(X(t;x))$. Since $b$ is smooth then $\alpha$ is locally Lipschitz, and \ref{ODE} admits the unique solution $\varphi=0$,
		completing the proof.
	\end{proof}
	\noindent
	Let us note that if $b$ and $f$ are continuous functions,  previous computations fail to be true, and the equation \eqref{pde} has to be understood in the {\it viscosity sense}, see below.
	\subsection{Viscosity solutions}
	 For the sake of completeness, let us introduce some notations that we will use throughout this section.
	 We indicate with  $USC([0,+\infty)\times\R^d)$  the space of upper semicontinuous functions on $[0,+\infty)\times\R^d$, while we use $LSC([0,+\infty)\times\R^d)$ for the space of lower semicontinuous functions on $[0,+\infty)\times\R^d$.\\
	 Given a function  $g: \R^d\rightarrow \R$, we say that
	 \begin{itemize}
	 	\item $g$ satisfies the H\"{o}lder condition if the following holds
	 	$$
	 	|g(x)-g(y)|\le C|x-y|^\beta\,,
	 	$$
	 	for $0<\beta\le 1$, $C>0$.
	 	\item$g$ is a Lipschitz function with sublinear growth if
	 	$$
	 	|g(x)|\le C(1+|x|^\theta)\,,\qquad |g(x)-g(y)|\le C|x-y|\,,
	 	$$
	 	for $C>0$, $0<\theta<1$;
	 	\item  $g$ is a continuous decreasing function if  $\forall x,y\in\R$ one has
	 	$$
		\langle{g(x)-g(y),x-y}\rangle\le0\;
	 	$$
	 	\end{itemize}
	 	Let $u:(0,+\infty)\times \mathbb{R}^d\rightarrow	\R$ be a continuous function. We define $\mathcal{D}^{1,+} u(t_0,x_0)$ the superdifferential of $u$ at $(t_0,x_0)$, i.e. the set of all points $(a,p)\in\R\times\R^d$ such that for $(t,x)\to (t_0,x_0)$ we have
	 	$$
	 	u(t,x)\le u(t_0,x_0) + a(t-t_0) + \langle p,x-x_0\rangle + o(|t-t_0|+|x-x_0|)\,.
	 	$$
	 	Moreover, we define $\mathcal{D}^{1,-} u(t_0,x_0)$ the subdifferential of $u$ at $(t_0,x_0)$, i.e. the set of all points $(a,p)\in\R\times\R^d$ such that for $(t,x)\to (t_0,x_0)$ we have
	 	$$
	 	u(t,x)\ge u(t_0,x_0) + a(t-t_0) + \langle p,x-x_0\rangle + o(|t-t_0|+|x-x_0|)\,.
	 	$$
Furthermore, we recall the definition of viscosity solutions given in \cite{lions}.
\begin{definition}\label{d1}
	\begin{itemize}
		\item[(i)] Let  $u\in USC([0,+\infty)\times\R^d)$ and be a bounded function from above. We say that $u$ is a viscosity subsolution of \eqref{pde} in $(0,+\infty)\times \R^d$ if 
		\begin{align*}
			\begin{cases}
				&\varphi_t(t,x)-b(x)\cdot\nabla\varphi(t,x)\le 0\,, \\
				&u(0,x)\le f(x),\ \  \ \ \ \  \ \ \ \ \forall x\in\R^d,
			\end{cases}
		\end{align*}
		for any $\varphi\in C^1 (\R^{d+1})$ such that $\varphi-u$ has a (strict) minimum value at $(t,x)\in (0,+\infty)\times \R^d$.
		\item[(ii)] Let $u\in LSC([0,+\infty)\times\R^d)$ and be a bounded function from below. We say that $u$ is a viscosity supersolution of \eqref{pde} in $(0,\infty)\times \R^d$ if 
		
		\begin{align*}
			\begin{cases}
				&	\varphi_t(t,x)-b(x)\cdot\nabla\varphi(t,x)\ge 0\,,\\
				&u(0,x)\ge f(x),\ \  \ \ \ \  \ \ \ \ \forall x\in\R^d,
			\end{cases}
		\end{align*}
		for any $\varphi\in C^1 (\R^{d+1})$ such that $\varphi-u$ has a (strict) maximum value at $(t,x)\in (0,+\infty)\times \R^d$.
		
		\item[(iii)] Let $u$ be a bounded continuous function. We say that $u$ is a viscosity solution of \eqref{pde} in $(0,+\infty)\times\R^d$ if it is both a subsolution and supersolution.
	\end{itemize}
\end{definition}
\begin{remark}
	As it is well know, Def. \ref{d1} can be expressed in terms of subdifferential and superdifferential, i.e.,
	\begin{align*}
		&a-b(x)\cdot p\le 0\,  \ \ \  \ \forall (t,x)\in(0,+\infty)\times\R^d, \ \forall (a,p)\in\mathcal{D}^{1,+} u(t,x), \\
		&a-b(x)\cdot p\ge 0\, \ \ \ \ \forall (t,x)\in(0,+\infty)\times\R^d, \ \forall (a,p)\in\mathcal{D}^{1,-} u(t,x).
	\end{align*}
\end{remark}

	\subsection{Existence and uniqueness results}
	To prove both existence and uniqueness of a solution $u$ as in \eqref{u}, we need to show the following
	\begin{theorem}[Comparison Principle]\label{seiscocciante}
		Let $u$, resp. $v$, be a subsolution, resp. a supersolution, of \eqref{pde}.
		Suppose that:
		\begin{enumerate}
			\item $f$ is bounded and satisfies the H\"{o}lder condition;
			\item at least one of the following conditions is satisfied
			\begin{itemize}
				\item [(i)] $b$ is a Lipschitz function with sublinear growth,
				\item [(ii)] $b$ is a continuous decreasing function.
			\end{itemize}
		\end{enumerate}
		Then $u(t,x)\le v(t,x)$ for all $(t,x)\in[0,+\infty)\times\R^d$.
	\end{theorem}
	\noindent
	\begin{remark}
		The proof is based on the ideas developed in \cite{crandall}. However, the parabolic case in the unbounded domain $\R^d$ was not developed there.  For the convenience of the reader we include the proof of Theroem \ref{seiscocciante} within the Appendix \ref{A1}.
	\end{remark}
	\noindent
	\begin{theorem}
		Suppose that $f$ and $b$ satisfies hypotheses of Thorem \ref{seiscocciante}.
		Then there exists a unique viscosity solution $u$ of the problem \eqref{pde}. Moreover, $u$ admits the following representation formula
		\begin{equation}\label{repru}
			u(t,x)=f(X(t;x))\,,
		\end{equation}
		where $X(t;x)$ solves \eqref{dyn}.
	\end{theorem}
	\begin{proof}
		The uniqueness directly follows from Thereom \ref{seiscocciante}. Concerning the existence, if $f$ and $b$ are smooth function, then the existence and the representation formula have been already proved in Propositon \ref{r1}.\\
		In the general case, let ${\{f_n\}}_n$ and ${\{b_n\}}_n$ be two approximating sequences  of smooth functions, locally uniformly converging to $f$ and $b$. Without loss of generality, we can suppose that $b_n$ and $f_n$ are respectively locally and globally bounded uniformly in $n$, and satisfy the same assumptions of $f$ and $b$ with  H\"{o}lder and Lipschitz constants bounded uniformly in $n$.
		
		Let $X_n$ be the solution of \eqref{dyn} associated to $b_n$. Let $u_n$ be the solution of \eqref{pde} with $f_n$ and 
		fix $K\subset\subset [0,+\infty)\times\R^d$. Since $b_n$ and $f_n$ are locally bounded uniformly in $n$, we have that for all $(t,x)\in K$
		\begin{align}
			&|X_n(t;x)|\le |x|+\int_0^t\left|b_n(X_n(s;x))\right|\,\mathrm{d}s\le C_K\,,\label{x1}\\
			&|u_n(t,x)|=|f_n(X_n(t;x))|\le C_K\,,
		\end{align}
		for a positive constant $C_K$.\\
		Now we want to prove the Lipschitz bounds for $X_n$ and the H\"{o}lder bounds for $u_n$.\\
		Using the estimate in \eqref{x1}, and recalling that $X_n$ is solution of \eqref{dyn} we have 
		\begin{align*}
			|X_n(t,x)-X_n(s,x)|\le\int_s^t|b_n(X_n(r;x))|\,\mathrm{d}r\le C_K|t-s|
		\end{align*}
		for any $(t,x), (s,x)\in K$ and $t\geq s$.\\
		Whereas for all $(t,x)$ and $(t,y)$ belong to  $K$, one has
		\begin{align*}
			|X_n(t,x)-X_n(t,y)|^2\le |x-y|^2+\int_0^t\left(b_n(X_n(r,x))-b_n(X_n(r,y))\right)\cdot(X_n(r,x)-X_n(r,y))\,\mathrm{d}r\,.
		\end{align*}
		If $b_n$ satisfies $(i)$, then we have
		\begin{align*}
			|X_n(t,x)-X_n(t,y)|^2\le |x-y|^2+C\int_0^t|X_n(r,x)-X_n(r,y)|^2\,\mathrm{d}r\,.
		\end{align*}
		By Gronwall's lemma, we have
		\begin{align*}
			|X_n(t,x)-X_n(t,y)|^2\le |x-y|^2 e^{Ct}\implies |X_n(t,x)-X_n(t,y)|\le C_K|x-y|\,,
		\end{align*}
		hence obtaining Lipschitz bounds for $X_n$.\\
		If $b_n$ satisfies $(ii)$, then
		$$
		\left(b_n(X_n(r,x))-b_n(X_n(r,y))\right)\cdot(X_n(r,x)-X_n(r,y))\le 0\implies  |X_n(t,x)-X_n(t,y)|\le C|x-y|\,,
		$$
		where $C$ is a positive constant.\\
		Then, by combining above estimates, we have
		\begin{equation}\label{stimaX}
			|X_n(t,x)-X_n(s,y)|\le C_K\left(|t-s|+|x-y|\right)\,,\qquad\forall (t,x), \ (s,y)\in K\,.
		\end{equation}
		Concerning the H\"{o}lder bounds for $u_n$, by
		the H\"{o}lder bound on $f_n$, together with the estimate \eqref{stimaX},  we have
		\begin{equation}\label{stimau}
			|u_n(t,x)-u_n(s,y)|\le C_1|X_n(t,x)-X_n(s,y)|^\beta\le C_K\left(|t-s|^\beta+|x-y|^\beta\right)\,.
		\end{equation}
		Therefore, $\{u_n\}_n$ and $\{X_n\}_n$ are uniformly bounded and uniformly equicontinous in all compact sets 
		$K\subset\subset[0,\infty)\times\R$. Using the Ascoli-Arzel\'a's theorem, we obtain 
		$$
		\exists\, u,\, X\quad\mbox{s.t.}\qquad u_n\to u\,,\quad X_n\to X\quad\mbox{locally uniformly}\,.
		$$
		Recalling that \eqref{repru} holds true for $u_n$, and by the uniform convergence of $f_n$, $X_n$ and $u_n$, we obtain that $u$ admits the representation formula \eqref{repru}.
		Therefore, we are left to prove that $u$ is a solution of \eqref{pde}. Indeed, we consider a smooth function $\phi$ such that $\phi-u$ achieves his (strict) minimum in $(t,x)\in(0,+\infty)\times\R^d$.
		Perturbing $\phi$ with a smooth function $\psi$ as follow:
		\begin{itemize}
			\item$ \phi(s,y)=\psi(s,y)$ for any $(s,y)\in B_1(t,x)$;
			\item $\lim\limits_{s\to+\infty}\psi(s,y)=\lim\limits_{|y|\to+\infty}\psi(s,y)=+\infty$;
			\item $\psi-u$ achieves his (strict) minimum in $(t,x)$.
		\end{itemize}
		and recalling that $u_n\rightarrow u$ uniformly, we derive that $\psi-u_n$ admits minimum, i.e.
		$$
		\min\limits_{(0,+\infty)\times\R^d}(\psi-u_n)=\psi(t_n,x_n)-u_n(t_n,x_n)\,.
		$$
		Since $\psi$ is coercive and $u_n$ is bounded, there exists a compact subeset $S$ of $(0,+\infty)\times\R^d$ such that the sequence $\{(t_n,x_n)\}_n\subset S$. Then, up to subsequences, we have that $(t_n,x_n)\to (\bt,\bx)$, for a certain $(\bt,\bx)\in S $.
		But since $(t_n,x_n)$ is the minimum for $\psi-u_n$, we have
		$$
		\psi(s,y)-u_n(s,y)\ge \psi(t_n,x_n)-u_n(t_n,x_n)\,, \forall (s,y)\in (0,+\infty)\times\R^d\,,
		$$
		and passing to the limit for $n\rightarrow +\infty$, we obtain
		$$
		\psi(s,y)-u(s,y)\ge \psi(\bt,\bx)-u(\bt,\bx)\,,
		$$
		meaning that $(\bt,\bx)$ is a global minimum for $\psi-u$, and so $(\bt,\bx)=(t,x)$.
		Using the definition of viscosity subsolution of \eqref{pde} for $u_n$, we have
		$$
		\psi_t(t_n,x_n)-b(x_n)\cdot\nabla\psi(t_n,x_n)\le 0\,.
		$$
		Again, passing to the limit and taking into account that $\psi$ and $\phi$ coincide in $B_1 (t,x)$, we get
		$$
		\phi_t(t,x)-b(x)\cdot\nabla\phi(t,x)\le 0\,,
		$$
		which proves that $u$ is a subsolution of \eqref{pde}. Analogously, we can  show that $u$ is also a supersolution, hence completing  the proof.
	\end{proof}
	
	\subsection{Asymptotic of viscosity solutions}
	In this section we study the asymptotic behaviour of $u$, starting by studying the asymptotic behaviour of $X$.
	\begin{proposition}\label{exX}
		Let $b$ be a continuous function satisfying hypotheses of Theorem \ref{seiscocciante}. Suppose that $X:(0,+\infty)\times \mathbb{R}^d\rightarrow \R^d $ is a solution of \eqref{dyn}, then the following holds true.
		\begin{itemize}
			\item If $b$ satisfies the following
			\begin{equation}\label{eq}
				b(x)\cdot(x-x_0)< 0\,
			\end{equation}
			for some $x_0\in\R^d$ and $x\in\R^d$, then $x_0$ is a globally asymptotically stable equilibrium point. 
			\item If $b$ is such that
			\begin{equation}\label{exphp}
				b(x)\cdot(x-x_0)\le -c|x-x_0|^2\,,
			\end{equation}
			for some positive constant $c$, then $X$ satisfies
			\begin{equation}\label{expdecay}
				|X(t;x)-x_0|\le |x-x_0|e^{-ct}\,  
			\end{equation}
			for all \ $t\in (0,+\infty), x\in \R^d $\,.
		\end{itemize}
	\end{proposition}
	\begin{proof}
		The proof is a standard application of both Lyapunov's Theorem and Gronwall's lemma.
		Suppose that $b$ satisfies \eqref{eq}, then evaluating \eqref{eq} in $x=x_0+re_i$, $r\in\R$, $i=1,\dots,d\,,$ we get
		$$
		re_i\,b(x_0+re_i)<0\quad\forall r\in\R\implies b(x_0+re_i)=0\quad\forall i\implies b(x_0)=0\,,
		$$
		which implies that $x_0$ is an equilibrium point. Moreover 
		$$
		V(x)=|x-x_0|^2\,,
		$$
		is a strict Lyapunov function, hence  $x_0$ is a globally asymptotically stable equilibrium point.
		Suppose that \eqref{exphp} holds true. Using the fundamental theorem of calculus, we have
		\begin{align*}
			|X(t;x)-x_0|^2&=|x-x_0|^2+2\int_0^t b(X(s;x))\cdot (X(s;x)-x_0)\,ds\\&\le |x-x_0|^2-2c\int_0^t|X(s;x)-x_0|^2\,ds\,,
		\end{align*}
		for some positive constant $c$ and by Gronwall's lemma we get
		\begin{equation*}
			|X(t;x)-x_0|^2\le |x-x_0|^2e^{-2ct} \,,
		\end{equation*}
		implying 
		\begin{equation*}
			|X(t;x)-x_0|\le |x-x_0|e^{-ct}\,.\qedhere
		\end{equation*}
	\end{proof}
	\noindent
	Proposition \ref{exX} implies
	the following convergence result for the function $u$.
	\begin{theorem}\label{t1}
		Suppose that $b$ satisfies the hypotheses of  Proposition \ref{exX} and let $f$ be a function such that the assumptions of Theorem \ref{seiscocciante} hold true. Then, the solution $u$ of \eqref{pde} satisfies
		\begin{equation}\label{convu}
			|u(t,x)-f(x_0)|\le C|x-x_0|^\beta e^{-c\beta t}\, \ \mbox{for all} \  t\in(0,+\infty),\ x\in\R^d\,,
		\end{equation}
		where $c$ and $C$ are positive constant, $0<\beta \leq 1$ and $x_0\in \mathbb{R}^d$. 
	\end{theorem}
	\begin{proof}
		By the H\"{o}lder assumption on $f$, recalling the representation \eqref{repru} and \eqref{expdecay}, we have
		$$
		|u(t,x)-f(x_0)|=|f(X(t;x))-f(x_0)|\le C|X(t;x)-x_0|^\beta\le C|x-x_0|^\beta e^{-c\beta t}\,,
		$$
		and the proof is complete.
	\end{proof}
	\subsection{Random time viscosity solution}
	Let  $(\Omega,\mathcal{F},\P)$ be a probability space and let $E(\cdot)$ be an inverse subordinator process. Given the solutions $u$, resp. $X$ of \eqref{pde}, resp. \eqref{dyn}, we can consider the  random time dynamic
	$$
	Y(t,\omega;x)=X(E(t,\omega);x)\,,\qquad t\in [0,+\infty), \ x \in\X , \  \omega \in \Omega \,,
	$$
	with corresponding function $v$
	\begin{equation}\label{v}
		v(t,x)=\E[u(E(t),x)]=\E[f(Y(t;x))]\,,
	\end{equation}
	where, without loss of generality, with $E(t)$, resp. $Y(t;x)$, we shortly refer to $E(t,\cdot)$, resp. to $Y(t,\cdot \ ;x)\,.$	\\
	According to what we have already observed  in Section \ref{randomds}, $v$ is the solution to an evolution equation with the same generator $L$.
	
	\subsubsection{General classes of random times}

	In the case of dissipative dynamic, we can easily obtain a general estimate for the function $v$.
	
	\begin{proposition}
		Assume that the assumptions
		of Lemma \ref{KL} and Theorem \ref{t1} are satisfied, with $f(x_0)=0$. Then
		\[
		|v(t,x)|\le C_x\frac{1}{\alpha|v|\Gamma(\gamma)}t^{\gamma-1}Q(t),\quad t\to\infty.
		\]
		
	\end{proposition}
	
	\begin{proof}
		
		From \eqref{convu} we have
		$$
		|u(t,x))|\le C|x-x_0|^\beta e^{-c\beta t}\,.
		$$
		Hence, using (\ref{SUBFOR}) we get
		$$
		|v(t,x)|\le C_x\int_{0}^{\infty}e^{-c\beta t}G_{t}(\tau)\,d\tau\,,
		$$
		and by Lemma \ref{KL} we conclude the proof.
	\end{proof}
	\begin{remark}
		For the $\alpha$-stable subordinator considered in
		Example \ref{exa:alpha-stable1}, we obtain $v(t,x)\sim Ct^{-\alpha}$,
		$C$ is a constant. Therefore, starting with a solution $u(t,x)$
		with exponential decay after subordination we observe a polynomial
		decay with the order defined by the random time characteristics.
	\end{remark}
		In the next Subsection \ref{cpp} we analyse the behavior of the subordination under a Compound Poisson Process (CPP).
	\subsubsection{The case of inverse Poisson process}\label{cpp}
	In the case of a CPP, we can not apply Lemma \ref{KL} since hypotheses \ref{hp} are not satisfied. Therefore, we need to change our approach.
	\begin{theorem}
		Suppose that the assumptions in Proposition \ref{exX} and Theorem \ref{seiscocciante} hold.
		Let $S(\cdot)$ be a CPP with rate $\lambda$. Assume that the jumps $R_i$ have a finite moment of order $\alpha>0$. Then the function $v$ satisfies
		$$
		|v(t,x)-f(x_0)|\le \frac{C}{t^\alpha}|x-x_0|^\beta\,,
		$$
		where $x_0\in\X$,  $C$ is a positive constant and $0<\beta\le 1$. \\
		Moreover, if there exists $\delta>0$ such that $\E\left[e^{\delta R}\right]<+\infty$, then $v$ satisfies
		$$
		|v(t,x)-f(x_0)|\le C|x-x_0|^\beta e^{-\eta t}\,,
		$$ 
		with $\eta >0$.
	\end{theorem}
	\begin{proof}
		By Theorem \eqref{t1} we get
		\begin{equation}\label{gabryponte}
			|v(t,x)-f(x_0)|\le C|x-x_0|^\beta\E\left[e^{-c\beta E(t)}\right]\,.
		\end{equation}
		In order to estimate the average term, we can argue as follows:
		\begin{align*}
			\E\left[e^{-c\beta E(t)}\right]=\int_0^{+\infty}\P\left(e^{-c\beta E(t)}\ge r\right)\mathrm{d}r =\int_0^1\P\left(E(t)\le -\frac{\ln r}{c\beta}\right)\mathrm{d}r = \int_0^1\P\left(S\left(-\frac{\ln r}{c\beta}\right)\ge t\right)\mathrm{d}r\,,
		\end{align*}
		since $-\ln r<0$ when $r>1$ and $\P(E(t)\le z)=\P(S(z)\ge t)$ for all $t,z>0\,.$
		To compute the last integral, we  separately study the quantity $\P(S(z)\ge t)$, for $t,z\ge 0$. We have
		\begin{align*}
			\P(S(z)\ge t)=\sum\limits_{k=0}^{+\infty}\P(S(z)\ge t | N(z)=k)\P(N(z)=k)=\mathlarger{\sum\limits_{k=1}^{+\infty}}\,\P\left(\sum\limits_{i=1}^k R_i \ge t\right) e^{-\lambda z}\frac{(\lambda z)^k}{k!} \,,
		\end{align*}
		which allows us to derive 
		\begin{equation}\label{fannobam}\begin{split}
				\E\left[e^{-c\beta E(t)}\right]&=\mathlarger{\int_0^1\sum\limits_{k=1}^{+\infty}}\left[\P\left(\sum\limits_{i=1}^k R_i \ge t\right)\, e^{\frac{\lambda\ln s}{c\beta}}\left(-\frac{\lambda\ln s}{c\beta}\right)^k\,\frac 1{k!}\right]\,ds\\
				&=\mathlarger{\sum\limits_{k=1}^{+\infty}}\left[ \frac 1{k!}\left(\frac{\lambda}{c\beta}\right)^k\P\left(\sum\limits_{i=1}^k R_i \ge t\right)\int_0^{+\infty} e^{-\frac{\lambda+c\beta}{c\beta}y}y^k\,dy \right]\,,
		\end{split}\end{equation}
		where we have used the change of variable $y=-\ln s\,.$ Exploiting the density of a $\Gamma\left(k+1,\frac{\lambda+c\beta}{c\beta}\right)$ function, we know that the latter integral equals
		$$
		\int_0^{+\infty} e^{-\frac{\lambda+c\beta}{c\beta}y}y^k\,dy=k!\left(\frac{c\beta}{\lambda+c\beta}\right)^{k+1}\,.
		$$
		Then, exploiting above estimates, we can rewrite \eqref{gabryponte} as follows
		\begin{equation}\label{gosens}
			|v(t,x)-f(x_0)|\le \frac{Cc\beta}{\lambda+c\beta}|x-x_0|^\beta\,\mathlarger{\sum\limits_{k=1}^{+\infty}}\left[\P\left(\sum\limits_{i=1}^k R_i \ge t\right)\left(\frac{\lambda}{\lambda+c\beta}\right)^k\right]\,.
		\end{equation}
		By assumption on $R_i$, and using Markov's inequality, we get
		\begin{equation}\label{go}
			\P\left(\sum\limits_{i=1}^k R_i \ge t\right)\le \frac{\E\left[\left(\sum\limits_{i=1}^k R_i\right)^\alpha\right]}{t^\alpha}\le \frac{1}{t^\alpha}\,k^{\alpha+1}\E[R^\alpha]\,.
		\end{equation}
		Replacing \eqref{go}  in \eqref{gosens}, and recalling that the series converge, we have
		$$
		|v(t,x)-f(x_0)|\le \frac C{t^\alpha}|x-x_0|^\beta\,\mathlarger{\sum\limits_{k=1}^{+\infty}}\left(k^{\alpha+1}\left(\frac{\lambda}{\lambda+c\beta}\right)^k\right)=\frac{C}{t^\alpha}|x-x_0|^\beta\,,
		$$
		where $C$ is a positive constant.
		Analogously, if $\E[e^{\delta R}]<+\infty$ for a certain $\delta>0$, fixing $0<\eta<\delta$, we can apply again Markov's inequality, with the function $\phi(x)=e^{\eta x}$, to get
		$$
		\P\left(\sum\limits_{i=1}^k R_i \ge t\right)\le e^{-\eta t}\,\E\left[e^{\eta\sum\limits_{i=1}^k R_i} \right]= e^{-\eta t}\,\E\left[e^{\eta R}\right]^k\,,
		$$
		implying that \eqref{gosens} becomes
		$$
		|v(t,x)-f(x_0)|\le Ce^{-\eta t}|x-x_0|^\beta\,\mathlarger{\sum\limits_{k=1}^{+\infty}}\left(\frac{\lambda\, \E[e^{\eta R}]}{\lambda+c\beta}\right)^k=Ce^{-\eta t}|x-x_0|^\beta\,,
		$$
		where we choose $\eta$ such that
		$$
		\frac{\lambda\, \E[e^{\eta R}]}{\lambda+c\beta}<1\,,
		$$
		which completes the proof.
	\end{proof}

	\section{Appendix: proof of Theorem \ref{seiscocciante}}\label{A1}
	
	We note that, with the change of variable $\tilde{u}(t,x)=e^{-\lambda t}u(t,x)$, the system \eqref{pde} is equivalent to the following one
	\begin{equation}\label{secondo}
		\begin{cases}
			u_t(t,x)-b(x)\cdot\nabla u(t,x)+\lambda u(t,x)=0\,,\\
			u(0,x)=f(x)\,.\\
		\end{cases}
	\end{equation}
	\noindent
	therefore, to show the validity of  the comparison principle, we will work with the problem \eqref{secondo}, for a certain $\lambda\gg1$ to be later chosen.	\\
	\noindent
	Let $u$, resp. $v$, be a subsolution, resp.  a supersolution, of \eqref{secondo}.
	Arguing by contradiction, let us assume that there exists $ (s,z)\in[0,+\infty)\times\R^d$ such that $u(s,z)-v(s,z)=\delta>0$. For $\alpha,\nu>0$, we consider
	\begin{equation}\label{quantitya}
		u(t,x)-v(t,y)-\frac\alpha 2|x-y|^2-\frac 1\alpha|x|^2-{\nu}{t}\,.
	\end{equation}
	Due to the coercive term $\frac 1\alpha|x|^2$ and the boundedness of $u$ and $v$, the \eqref{quantitya} achieves a maximum.\\
	We denote the maximum by $M$ and one of its maximum points with $(\bar{t},\bar{x},\bar{y})\in[0,+\infty)\times\R^{2d}$.
	Hence, for $\nu$ sufficiently small and $\alpha$ sufficiently large, we have
	\begin{equation}\label{scepsi}
		u(\bt,\bx)-v(\bt,\by)-\frac\alpha 2|\bx-\by|^2-\frac 1\alpha|\bx|^2-{\nu}{\bt}\ge u(s,z)-v(s,z)-\frac 1\alpha|z|^2-\nu s\ge\delta-\frac\delta 2=\frac\delta 2\,.
	\end{equation}
	Moreover, thanks to the boundedness of $u$ and $v$, it holds
	\begin{equation}\label{diego}
		\frac 1\alpha|\bx|^2+\frac\alpha 2|\bx-\by|^2\le u(\bt,\bx)-v(\bt,\by)\le C\implies \lim\limits_{\alpha\to\infty}|\bx-\by|=\lim\limits_{\alpha\to\infty}\frac{|\bx|}{\alpha}^{1+\theta}\!\!\!\!=0\,. 
	\end{equation}
	\underline{Case 1} If $\bar{t}=0$, one has
	$$
	\frac\delta 2\le u(s,z)-v(s,z)-\frac 1\alpha|z|^2-{\nu}{s}\le u(0,\bx)-v(0,\by)-\frac\alpha 2|\bx-\by|^2-\frac 1\alpha|\bx|^2\,.
	$$
	Since $f$ is H\"{o}lder continuous, $u(0,\cdot)\le f(\cdot)\le v(0,\cdot)$ and $|\bx-\by|\to0$, we have that for $\alpha$ sufficiently large
	$$
	u(0,\bx)-v(0,\by)\le\frac{\delta}{3}\implies u(0,\bx)-v(0,\by)-\frac\alpha 2|\bx-\by|^2-\frac 1\alpha|\bx|^2\le\frac\delta 3\,,
	$$
	which leads to a contradiction. \\
	\underline{Case 2}: Suppose that $\bar{t}$ belongs to $\in(0,+\infty)$. Using \cite[Theorem 8.3.]{crandall} with the following choices
	\begin{align*}
		&u_1(t,x)=u(t,x)\,,\quad u_2(t,y)=-v(t,y)\,,\\ &\phi(t,x,y)=\frac\alpha 2|x-y|^2+\frac 1\alpha|x|^2+{\nu}{t}\,
	\end{align*}
	then there exist $a,c\in \R$ such that $a+c=\nu\,$ and 
	$$
	(a,\nabla_x\phi(\bar{t},\bar{x},\bar{y}))\in\overline{\mathcal{D}^{1,+}u(\bar{t},\bar{x})}\,,\qquad(-c,-\nabla_y\phi(\bar{t},\bar{x},\bar{y}))\in\overline{\mathcal{D}^{1,-}v(\bar{t},\bar{y})}\,.
	$$
	From now on, we will omit the dependence on $(\bar{t},\bar{x},\bar{y})$ for the function $\phi$. Since $u$ is a subsolution, while $v$ is a supersolution, of \eqref{secondo}, then we have
	\begin{align*}
		&a-b(\bx)\cdot \nabla_x\phi+\lambda u(\bar{t},\bar{x})\le 0\,,\\
		-&c+b(\by)\cdot \nabla_y\phi+\lambda v(\bar{t},\bar{y})\ge 0\,.
	\end{align*}
	Subtracting the two inequalities we obtain
	\begin{align}\label{qui}
		\nu+\lambda(u(\bar{t},\bar{x})-v(\bar{t},\bar{y}))\le b(\bx)\cdot\nabla_x\phi+b(\by)\cdot \nabla_y\phi\,.
	\end{align}
	The first term in the left-hand side is non-negative, so we can ignore it. For the second term, using \eqref{scepsi} we get
	\begin{align*}
		\lambda\frac\delta 2+\lambda\frac{\alpha}{2}|\bx-\by|^2\le \,b(\bx)\cdot\nabla_x\phi+b(\by)\cdot \nabla_y\phi\,.
	\end{align*}
	To estimate the right-hand side term, we first compute the derivatives of $\phi$ 
	\begin{align*}
		\nabla_x\phi=\alpha(\bar{x}-\bar{y})+\frac 2\alpha\bx\,,\qquad\nabla_y\phi=-\alpha(\bar{x}-\bar{y})\,,
	\end{align*}
therefore we have
	\begin{align*}
		b(\bx)\cdot\nabla_x\phi + b(\by)\cdot\nabla_y\phi \le\,\alpha(b(\bx)-b(\by))\cdot(\bx-\by)+\frac 2\alpha b(\bx)\cdot\bx\,.
	\end{align*}
	If $b$ satisfies the condition $(i)$, using \eqref{diego} we get
	$$
	\alpha(b(\bx)-b(\by))\cdot(\bx-\by)+\frac 2\alpha b(\bx)\cdot\bx\le C\alpha|x-y|^2+\omega(\alpha)\,,
	$$
	where $\omega(\alpha)$ is a quantity which tend to $0$ when $\alpha\to+\infty$.\\
	On the other hand, if $b$ satisfies the condition $(ii)$ we get
	$$
	\alpha(b(\bx)-b(\by))\cdot (\bx-\by)\le 0\,,\qquad b(\bx)\cdot\bx\le b(0)\cdot\bx\,,
	$$
	hence
	$$
	\alpha(b(\bx)-b(\by))\cdot(\bx-\by)+\frac 2\alpha b(\bx)\cdot\bx\le\omega(\alpha)\,.
	$$
	In both cases we obtain
	$$
	b(\bx)\cdot\nabla_x\phi + b(\by)\cdot\nabla_y\phi\le C\alpha|\bx-\by|^2+\omega(\alpha)\,.
	$$
	Using the above estimates in \eqref{qui}, we get
	$$
	\lambda\frac\delta 2 +\lambda\frac\alpha 2|\bx-\by|^2\le C\alpha|\bx-\by|^2+\omega(\alpha)\,,
	$$
	which again leads to a contradiction for $\alpha$ sufficiently small and $\lambda$ sufficiently large. Since $u$ and $v$ are, respectively, a subsolution and a supersolution of \eqref{secondo} for $\lambda\ge 0$, then $e^{\lambda t}u$, resp.  $e^{\lambda t}v$, is a subsolution, resp. a supersolution, of \eqref{pde}.
	Therefore, $e^{(\lambda-\mu)t}u$, resp. $e^{(\lambda-\mu)t}v$, is a subsolution, resp. a supersolution, of \eqref{secondo} with $\lambda$ replaced by $\mu$. 
	Then, considering $\mu$ large enough we have
	$$
	e^{(\lambda-\mu)t}u\le e^{(\lambda-\mu)t}v\implies u\le v\,,
	$$
	completing the proof.
	
	\subsection*{Acknowledgements}
	
	This work has been partially supported by Center for Research in Mathematics
	and Applications (CIMA) related with the Statistics, Stochastic Processes
	and Applications (SSPA) group, through the grant \\ UIDB/MAT/04674/2020
	of FCT-Funda{\c c\~a}o para a Ci{\^e}ncia e a Tecnologia, Portugal.
	The financial support by the Ministry for Science and Education of
	Ukraine through Project 0119U002583 is gratefully acknowledged.


\begin{thebibliography}{abc99xyz} 
		\bibitem{Atanackovic2009}
		T.~M. Atanackovic, S.~Pilipovic, and D.~Zorica.
		\newblock {Time distributed-order diffusion-wave equation. I., II.}
		\newblock In {\em Proceedings of the Royal Society of London A: Mathematical,
			Physical and Engineering Sciences}, volume 465, pages 1869--1891. The Royal
		Society, 2009.
		\bibitem{barles} Barles, G., Biton, S., Bourgoing, M., Ley, O.  {\it Uniqueness Results for Quasilinear Parabolic Equations through Viscosity Solutions' Methods}. Calc. Var. 18, 159179, 2003.
		\bibitem{Bertoin96}
		J.~Bertoin.
		\newblock {\em L\'evy processes}, volume 121 of {\em Cambridge Tracts in
			Mathematics}.
		\newblock Cambridge University Press, Cambridge, 1996.
		
		\bibitem{Blumenthal1968}
		R.~M. Blumenthal and R.~K. Getoor.
		\newblock {\em {Markov Processes and Potential Theory}}.
		\newblock Academic Press, 1968.
		
		\bibitem{Bingham1971}
		N.~H. Bingham.
		\newblock Limit theorems for occupation times of {M}arkov processes.
		\newblock {\em Z.\ Wahrsch.\ verw.\ {G}ebiete}, 17:1--22, 1971.
		
		\bibitem{Bochner1962}
		S.~Bochner.
		\newblock Subordination of non-{G}aussian stochastic processes.
		\newblock In {\em Proc.\ Natl.\ Acad.\ Sci.\ USA}, volume~4, pages 19--22.
		National Acad Sciences, 1962.
		
		\bibitem{Chen2017}
		Z.-Q. Chen.
		\newblock Time fractional equations and probabilistic representation.
		\newblock {\em Chaos Solitons Fractals}, 102:168--174, 2017.
		\bibitem{crandall} M.G. Crandall, H. Ishii, P.-L. Lions: \textit{User's guide to viscosity solutions of second order partial differential equations.} Bull. Amer. Soc., 27, 1-67, 1992.
		
		\bibitem{lions} Crandall, M.G., Lions, P.-L. {\it Viscosity Solutions of Hamilton-Jacobi Equations}.  Transactions of the American Mathematical Society, Vol. 277, Number 1, 1983.
		\bibitem{Daftardar-Gejji2008}
		V.~Daftardar-Gejji and S.~Bhalekar.
		\newblock Boundary value problems for multi-term fractional differential
		equations.
		\newblock {\em J. Math. Anal. Appl.}, 345(2):754--765, 2008.
		
		\bibitem{EN2000}
		K.~J. Engel and R.~Nagel.
		\newblock {\em One-parameter semigroups for linear evolution equations}.
		\newblock Graduate texts in mathematics. Springer, 2000.
		
		\bibitem{Feller71}
		W.~Feller.
		\newblock {\em An introduction to probability theory and its applications.
			{V}ol.\ {II}.}
		\newblock Second edition. John Wiley \& Sons Inc., New York, 1971.
			\bibitem{GS74}
		I.~I. Gihman and A.~V. Skorokhod.
		\newblock {\em The Theory of Stochastics Processes I, II}.
		\newblock Springer-Verlag, 1974.
		\bibitem{Gorenflo1999}
		R.~Gorenflo, Y.~Luchko, and F.~Mainardi.
		\newblock {Analytical properties and applications of the Wright function}.
		\newblock {\em Fract.\ Calc.\ Appl.\ Anal.}, 2(4):383--414, 1999.
		\bibitem{Gorenflo2005}
		R.~Gorenflo and S.~Umarov.
		\newblock {Cauchy and nonlocal multi-point problems for distributed order
			pseudo-differential equations, Part one}.
		\newblock {\em Z.\ Anal.\ Anwend.}, 24(3):449--466, 2005.
		\bibitem{GR81}
		I.~S. Gradstein and I.~M. Ryshik.
		\newblock {\em Tables of Series, Products and Integrals}.
		\newblock Academic Press, 225 Wyman Street, Waltham, MA 02451, USA, 8 edition,
		2015.
		
	
		
		
		
		\bibitem{Hanyga2007}
		A.~Hanyga.
		\newblock Anomalous diffusion without scale invariance.
		\newblock {\em J.\ Phys.\ A: Mat.\ Theor.}, 40(21):5551, 2007.
		
			\bibitem{Kochubei2008}
		A.~N. Kochubei.
		\newblock Distributed order calculus and equations of ultraslow diffusion.
		\newblock {\em J.\ Math.\ Anal.\ Appl.}, 340(1):252--281, 2008.
		
		\bibitem{Kochubei11}
		A.~N. Kochubei.
		\newblock General fractional calculus, evolution equations, and renewal
		processes.
		\newblock {\em Integral Equations Operator Theory}, 71(4):583--600,
		2011.
		\bibitem{KocKon2017}
		A.~N. Kochubei and Y.~G. Kondratiev.
		\newblock Fractional kinetic hierarchies and intermittency.
		\newblock {\em Kinet.\ Relat.\ Models}, 10(3):725--740, 2017.
		
		\bibitem{KKS2018}
		A.~Kochubei, Yu.~G. Kondratiev, and J.~L. da~Silva.
		\newblock From random times to fractional kinetics.
		\newblock {\em Interdisciplinary Studies of Complex Systems}, 16:5--32, 2020.
		
		\bibitem{KKdS19}
		A.~Kochubei, Yu.~G. Kondratiev, and J.~L. da~Silva.
		\newblock Random time change and related evolution equations. {T}ime asymptotic
		behavior.
		\newblock {\em Stochastics and Dynamics}, 4:2050034--1--24, 2020.
		
	
			\bibitem{KdS2020}
		Yu.~G. Kondratiev and J.~L. da~Silva.
		\newblock {Green Measures for Markov Processes}.
		\newblock {\em Methods Funct.\ Anal.\ Topology}, 26(3):241--248, 2020.
		
		\bibitem{KdS20}
		Yu.~G. Kondratiev and J.~L. da~Silva.
		\newblock {Green Measures for Time Changed Markov Processes}.
		\newblock {\em Acceped for publication in Methods of Funct. Anal. Topology}, 2021.
		\newblock arXiv:2008.03390.
	
			\bibitem{Magdziarz2015}
		M.~Magdziarz and R.~L. Schilling.
		\newblock {Asymptotic properties of Brownian motion delayed by inverse
			subordinators}.
		\newblock {\em Proceedings of the American Mathematical Society},
		143(10):4485--4501, 2015.
		\bibitem{Meerschaert2006}
		M.~M. Meerschaert and H.-P. Scheffler.
		\newblock Stochastic model for ultraslow diffusion.
		\newblock {\em Stochastic Process.\ Appl.}, 116(9):1215--1235, 2006.
		
		\bibitem{Meerschaert2012}
		M.~M. Meerschaert and A.~Sikorskii.
		\newblock {\em {Stochastic Models for Fractional Calculus}}, volume~43.
		\newblock Walter de Gruyter, 2012.
		
	
		
		\bibitem{Montroll1965}
		E.W. Montroll and G.H. Weiss.
		\newblock {Random Walks on Lattices. II}.
		\newblock {\em J.\ Math.\ Phys.}, 6(2):167--181, December 2004.
		
		\bibitem{RS75}
		M.~Reed and B.~Simon.
		\newblock {\em Methods of Modern Mathematical Physics}, volume~I.
		\newblock Academic Press, Inc., New York and London, 1975.
		
		\bibitem{Seneta1976}
		E.~Seneta.
		\newblock {\em Regularly Varying Functions}, volume 508 of {\em Lect.\ Notes
			Math.}
		\newblock Springer, 1976.
		\bibitem{Toaldo2015}
		B.~Toaldo.
		\newblock Convolution-type derivatives, hitting-times of subordinators and
		time-changed {$C_0$}-semigroups.
		\newblock {\em Potential Anal.}, 42(1):115--140, 2015.
		
		\bibitem{Yosida80}
		K.~Yosida.
		\newblock {\em Functional Analysis}.
		\newblock Springer-Verlag, Berlin Heidelberg New York, 6 edition, 1980.
		
		
		
		
		
	\end{thebibliography}
\end{document}